\documentclass[letterpaper,11pt,twoside,keywordsasfootnote,addressasfootnote,noinfoline]{article}
\usepackage{fullpage}
\usepackage[english]{babel}
\usepackage{amssymb}
\usepackage{amsmath}
\usepackage{theorem}
\usepackage{epsfig}
\usepackage{enumitem}
\usepackage{mathrsfs}
\usepackage{imsart}
\usepackage{color}

\theorembodyfont{\slshape}
\newtheorem{theorem}{Theorem}

\newtheorem{lemma}{Lemma}

\newenvironment{proof}{\noindent{\scshape Proof.}}{\hspace*{2mm} $\square$}

\newcommand{\G}{\mathscr{G}}
\newcommand{\V}{\mathscr{V}}
\newcommand{\E}{\mathscr{E}}

\newcommand{\C}{\mathscr{C}}

\newcommand{\ind}{\mathbf{1}}

\newcommand{\n}{\hspace*{-6pt}}

\DeclareMathOperator{\card}{card \,}

\DeclareMathOperator{\binomial}{Bin \,}

\DeclareMathOperator{\var}{Var}

\begin{document}


\begin{frontmatter}
\title     {First and second moments of the size distribution of \\ bond percolation clusters on regular graphs}
\runtitle  {Bond percolation clusters on regular graphs}
\author    {Nicolas Lanchier\thanks{Nicolas Lanchier was partially supported by NSF grant CNS-2000792.} and Axel La Salle}
\runauthor {Nicolas Lanchier and Axel La Salle}
\address   {School of Mathematical and Statistical Sciences, Arizona State University, Tempe, AZ 85287, USA. \\ nicolas.lanchier@asu.edu. alasall1@asu.edu}

\maketitle

\begin{abstract} \ \
 Motivated by network resilience and insurance premiums in the context of cyber security, we derive universal upper bounds for the first and second moments of the size of bond percolation clusters on finite regular graphs.
 Thinking of the clusters as dynamical objects coupled with branching processes gives a first set of bounds that are accurate when the probability of an edge being open is small.
 Estimating the number of isolated vertices, we also obtain a second set of bounds that are accurate when the probability of an edge being closed is small.
 As an illustration, we apply our results to the first three Platonic solids.
\end{abstract}


\end{frontmatter}


\section{Introduction}
\label{sec:intro}
 Bond percolation consists of a collection of independent Bernoulli random variables with the same success probability~$p$ indexed by the edges of a graph, with the edges associated to a success being open and the ones associated to a failure being closed.
 The percolation clusters are the connected components of the subgraph induced by the open edges.
 This process was introduced by~\cite{broadbent_hammersley_1957} to study the spread of a fluid through a medium. \\
\indent Bond percolation has been extensively studied on infinite graphs such as integer lattices~\cite{grimmett_1999} in which case the process typically exhibits a phase transition for the density~$p$ of open edges from a subcritical phase where all the percolation clusters are finite to a supercritical phase where at least one percolation cluster is infinite.
 Bond percolation has also been studied along increasing sequences of finite graphs: as the size of the graphs tends to infinity, the supercritical phase is now characterized by the existence of a giant connected component of open edges whose size scales like the size of the graph.
 Important examples are the complete graph, in which case the set of open edges consists of the~Erd\H{o}s-R\'{e}nyi random graph~\cite{erdos_renyi_1959}, and the hypercube~\cite{ajtai_komlos_szemeredi_1982}. \\
\indent Much less attention has been paid to bond percolation on fixed finite graphs in spite of its growing importance in terms of applications.
 Indeed, the first moment of the size of the percolation clusters on finite graphs is closely related to the notion of network resilience in computer network theory~\cite{kott_linkov_2019}.
 Similarly, the modeling framework introduced by~\cite{jevtic_lanchier_2020} shows that, in the context of cyber security, both the first and the second moments of the cluster size are keys to computing insurance premiums.
 Motivated by these aspects, \cite{jevtic_lanchier_lasalle_2020} studied the first and second moments of the cluster size on elementary graphs: the path, the ring and the star.
 In both contexts (network resilience and cyber insurance), the underlying graph represents a local area network, i.e., a finite group of computers~(the vertices) along with the way these computers are connected~(the edges).
 In this work, we study the first and second moments of the cluster size on general finite regular graphs that model local area networks more realistically than paths, rings or stars.


\section{Model description}
\label{sec:model}
 Throughout this paper, $\G = (\V, \E)$ is a finite connected~$D$-regular graph with~$N$ vertices.
 Let~$x$ be a vertex chosen uniformly at random and assume that the edges are independently open with probability~$p$.
 The main objective is to study the first and second moments of
 $$ S = \card (\C_x) \quad \hbox{where} \quad \C_x = \{y \in \V : \hbox{there is a path of open edges connecting~$x$ and~$y$} \}, $$
 the random number of vertices in the percolation cluster starting at~$x$.


\section{Coupling with a branching process}
\label{sec:branching}
 In this section, we prove the following upper bounds for the first and second moments.
\begin{theorem} -- Let~$\nu = (D - 1) p$ and~$R = N - 1$. Then,
\label{th:branching}
 $$ \begin{array}{rcl}
        E (S) & \n \leq \n & \displaystyle 1 + Dp \bigg(\frac{1 - \nu^R}{1 - \nu} \bigg) \vspace*{8pt} \\
      E (S^2) & \n \leq \n & \displaystyle \bigg(1 + Dp \bigg(\frac{1 - \nu^R}{1 - \nu} \bigg) \bigg)^2 + \frac{Dp (1 - p)}{(1 - \nu)^2} \bigg(\frac{(1 - \nu^R)(1 + \nu^{R + 1})}{1 - \nu} - 2R \nu^R \bigg). \end{array} $$
\end{theorem}
 To prove the theorem, the idea is to think of the cluster~$\C_x$ as a dynamical object described by a birth process starting with one particle at~$x$ and in which particles give birth with probability~$p$ onto vacant adjacent vertices.
 The size of the cluster is equal to the ultimate number of particles in the birth process which, in turn, is dominated stochastically by the number of individuals up to generation~$\card (\V) - 1 = N - 1$ in a certain branching process. \vspace*{4pt}


\noindent{\bf Birth process.}
 Having a vertex~$x \in \V$ and a realization of bond percolation with parameter~$p$ on the graph, we consider the following discrete-time birth process~$(\xi_n)$.
 The state at time~$n$ is a spatial configuration of particles on the vertices:
 $$ \xi_n \subset \V \quad \hbox{where} \quad \xi_n = \hbox{set of vertices occupied by a particle at time~$n$}. $$
 The process starts at generation~0 with one particle at~$x$, therefore~$\xi_0 = \{x \}$. Then,
\begin{itemize}
 \item for each vertex~$y$ adjacent to vertex~$x$, the particle at~$x$ gives birth to a particle sent to vertex~$y$ if and only if the edge~$(x, y)$ is open.
\end{itemize}
 These are the particles of generation~1.
 Assume the process has been defined up to generation~$n > 0$, and let~$Y_n = \card (\xi_n \setminus \xi_{n - 1})$ be the number of particles of generation~$n$.
 Label~$1, 2, \ldots, Y_n$ these particles and let~$x_{n, 1}, x_{n, 2}, \ldots, x_{n, Y_n}$ be their locations so that
 $$ \xi_n \setminus \xi_{n - 1} = \{x_{n, 1}, x_{n, 2}, \ldots, x_{n, Y_n} \}. $$
 Then, generation~$n + 1$ is defined sequentially from step~1 to step~$Y_n$ where, at step~$i$,
\begin{itemize}
 \item for each vertex~$y$ adjacent to vertex~$x_{n, i}$, the~$i$th particle of generation~$n$ gives birth to a particle sent to vertex~$y$ if and only if vertex~$y$ is empty and the edge~$(x_{n, i}, y)$ is open.
\end{itemize}
 Note that two particles~$i$ and~$j$ with~$i < j$ might share a common neighbor~$y$ in which case a child of particle~$i$ sent to~$y$ prevents particle~$j$ from giving birth onto~$y$.
 For a construction of the birth process from a realization of bond percolation on the dodecahedron, we refer to Figure~\ref{fig:birth}.
 The process is designed so that particles ultimately occupy the open cluster starting at~$x$.
 In particular, the total number of particles equals the cluster size, as proved in the next lemma.
\begin{figure}[t!]
\centering
\scalebox{0.36}{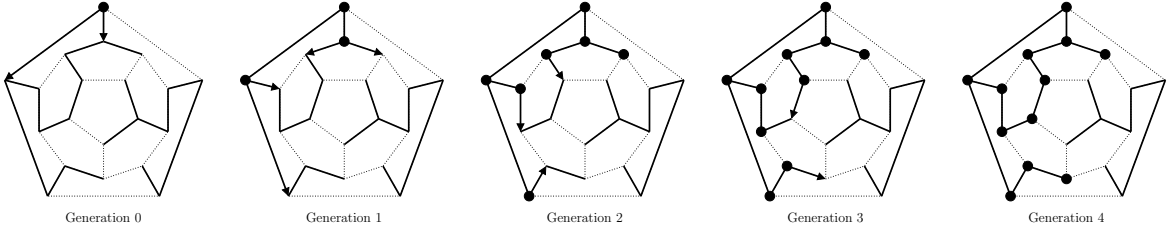}
\caption{\upshape{Example of a construction of the birth process from a realization of bond percolation on the dodecahedron.
                  The thick lines represent the open edges, the black dots represent the vertices occupied by a particle at each generation, and the arrows represent the birth events, from parent to children.}}
\label{fig:birth}
\end{figure}
\begin{lemma} --
\label{lem:wet-particle}
 We have~$S = \card (\xi_{N - 1}) = Y_0 + Y_1 + \cdots + Y_{N - 1}$.
\end{lemma}
\begin{proof}
 Because particles can only send a child to an empty vertex, each vertex is ultimately occupied by at most one particle.
 Also, the set occupied by a particle of generation~$n$ is
 $$ \xi_n \setminus \xi_{n - 1} = \{y \in \C_x : \hbox{the shortest open path connecting~$x$ and~$y$ has length~$n$} \}. $$
 In particular, all the vertices in the open cluster~$\C_x$ are ultimately occupied by exactly one particle whereas the vertices outside the cluster remain empty therefore
\begin{equation}
\label{eq:wet-particle-1}
  S = \card (\xi_0) + \card \bigg(\bigcup_{n = 1}^{\infty} \,(\xi_n \setminus \xi_{n - 1}) \bigg)
    = \card (\xi_0) + \sum_{n = 1}^{\infty} \,\card (\xi_n \setminus \xi_{n - 1}) = \sum_{n = 0}^{\infty} \,Y_n.
\end{equation}
 In addition, because the graph has~$N$ vertices, the shortest self-avoiding path on this graph must have at most~$N - 1$ edges, from which it follows that
\begin{equation}
\label{eq:wet-particle-2}
  \xi_n = \xi_{n - 1} \quad \hbox{and} \quad Y_n = \card (\xi_n \setminus \xi_{n - 1}) = 0 \quad \hbox{for all} \quad n \geq N.
\end{equation}
 Combining~\eqref{eq:wet-particle-1} and~\eqref{eq:wet-particle-2} gives the result.
\end{proof} \\


\noindent{\bf Coupling with a branching process.}
 The next step is to compare the number of particles in the birth process with the number of individuals in the branching process~$(X_n)$ where
 $$ X_0 = 1 \quad \hbox{and} \quad X_{n + 1} = X_{n, 1} + X_{n, 2} + \cdots + X_{n, X_n} \quad \hbox{for all} \quad n \geq 0 $$
 with the random variables~$X_{n, i}$ representing the number of offspring of individual~$i$ at time~$n$ being independent with probability mass function
 $$ X_{0, 1} = \binomial (D, p) \quad \hbox{and} \quad X_{n, i} = \binomial (D - 1, p) \quad \hbox{for all} \quad n, i \geq 1. $$
 This branching process can be visualized as the number of particles in the birth process above modified so that births onto already occupied vertices are allowed.
 In particular, the branching process dominates stochastically the birth process.
\begin{lemma} --
\label{lem:branching-particle}
 For all~$n \geq 0$, we have the stochastic domination~$Y_n \preceq X_n$.
\end{lemma}
\begin{proof}
 As for the branching process, for all~$n \geq 0$ and~$i \leq Y_n$, we let
 $$ Y_{n, i} = \hbox{\# offspring of the~$i$th particle of generation~$n$ in the birth process}. $$
 Because the edges are independently open with probability~$p$ and there are exactly~$D$ edges starting from each vertex, the number of offspring of the first particle is~$Y_1 = Y_{0, 1} = \binomial (D, p)$.
 For each subsequent particle, say the particle located at~$z$, we distinguish two types of edges starting from vertex~$z$ just before the particle gives birth.
\begin{itemize}
 \item There are~$m$ edges~$(z, y)$ that are connected to an occupied vertex~$y$.
       Because parent and offspring are located on adjacent vertices, we must have~$m \geq 1$. \vspace*{4pt}
 \item There are~$D - m$ edges~$(z, y)$ that are connected to an empty vertex~$y$.
       These edges have not been used yet in the construction of the birth process therefore each of these edges is open with probability~$p$ independently of the past of the process.
\end{itemize}
 From the previous two properties, we deduce that, for all~$n > 0$ and~$i \leq Y_n$,
\begin{equation}
\label{eq:branching-particle-2}
\begin{array}{rcl}
  P (Y_{n, i} \geq k) =
  E (P (Y_{n, i} \geq k \,| \,Y_{0, 1}, \ldots, Y_{n, i - 1})) \leq
  P (\binomial (D - 1, p) \geq k) = P (X_{n, i} \geq k). \end{array}
\end{equation}
 The stochastic domination follows from~$Y_1 = \binomial (D, p)$ and~\eqref{eq:branching-particle-2}.
\end{proof} \\


\noindent{\bf Number of individuals.}
 It directly follows from Lemmas~\ref{lem:wet-particle} and~\ref{lem:branching-particle} that
\begin{equation}
\label{eq:wet-branching}
  E (S^k) = E ((Y_0 + Y_1 + \cdots + Y_{N - 1})^k) \leq E ((X_0 + X_1 + \cdots + X_{N - 1})^k)
\end{equation}
 for all~$k > 0$.
 In view of~\eqref{eq:wet-branching}, the last step to complete the proof is to show that the upper bounds in the theorem are in fact the first and second moments of the total number of individuals up to generation~$R = N - 1$ in the branching process:
 $$ E (\bar X_R) \quad \hbox{and} \quad E (\bar X_R^2) \quad \hbox{where} \quad \bar X_R = X_0 + X_1 + \cdots + X_R. $$
\begin{lemma} --
\label{lem:branching-first}
 Let~$\nu = (D - 1) p$. Then,
 $$ E (\bar X_R) = 1 + Dp \bigg(\frac{1 - \nu^R}{1 - \nu} \bigg) \quad \hbox{for all} \quad R > 0. $$
\end{lemma}
\begin{proof}
 For~$i = 1, 2, \ldots, X_1$, let
 $$ \begin{array}{rcl}
    \bar Z_i & \n = \n & \hbox{number of descendants of the~$i$th offspring of the first individual} \vspace*{2pt} \\ &&
                         \hbox{up to generation~$R$, including the offspring}. \end{array} $$
 Then~$\bar X_R = 1 + \bar Z_1 + \cdots + \bar Z_{X_1}$ and the~$\bar Z_i$ are independent of~$X_1$ therefore
 $$ E (\bar X_R) = E (E (\bar X_R \,| \,X_1)) = E (1 + X_1 E (\bar Z_i)) = 1 + E (X_1) E (\bar Z_i) = 1 + Dp E (\bar Z_i). $$
 Because~$\bar Z_i$ is the number of individuals up to generation~$R - 1$ in a branching process with offspring
 distribution~$\binomial (D - 1, p)$, we deduce from~\cite[Theorem~2]{jevtic_lanchier_2020} that
 $$ E (\bar X_R) = 1 + Dp \bigg(\frac{1 - (\mu p)^R}{1 - \mu p} \bigg)
                 = 1 + Dp \bigg(\frac{1 - \nu^R}{1 - \nu} \bigg) \quad \hbox{where} \quad \nu = \mu p = (D - 1) p. $$
 This completes the proof.
\end{proof}
\begin{lemma} --
\label{lem:branching-second}
 Let~$\nu = (D - 1) p$. Then, for all~$R > 0$,
 $$ E (\bar X_R^2) = \bigg(1 + Dp \bigg(\frac{1 - \nu^R}{1 - \nu} \bigg) \bigg)^2 +
                     \frac{Dp (1 - p)}{(1 - \nu)^2} \bigg(\frac{(1 - \nu^R)(1 + \nu^{R + 1})}{1 - \nu} - 2R \nu^R \bigg). $$
\end{lemma}
\begin{proof}
 Using again~$\bar X_R = 1 + \bar Z_1 + \cdots + \bar Z_{X_1}$ and independence, we get
\begin{equation}
\label{eq:branching-second-1}
  \begin{array}{rcl}
    E (\bar X_R^2) & \n = \n & E (E ((1 + \bar Z_1 + \cdots + \bar Z_{X_1})^2 \,| \,X_1)) \vspace*{4pt} \\
                   & \n = \n & 1 + 2 E (X_1) E (\bar Z_i) + E (X_1) E (\bar Z_i^2) + E (X_1 (X_1 - 1))(E (Z_i))^2. \end{array}
\end{equation}
 In addition, using that~$X_1 = \binomial (D, p)$, we get
\begin{equation}
\label{eq:branching-second-2}
  \begin{array}{rcl}
    E (X_1 (X_1 - 1)) & \n = \n & \var (X_1) + (E (X_1))^2 - E (X_1) = Dp (1 - p) + D^2 p^2 - Dp = D (D - 1) p^2. \end{array}
\end{equation}
 Combining~\eqref{eq:branching-second-1} and~\eqref{eq:branching-second-2} and using some basic algebra give
 $$ \begin{array}{rcl}
      E (\bar X_R^2) & \n = \n & 1 + 2 Dp E (\bar Z_i) + Dp E (\bar Z_i^2) + D (D - 1) p^2 (E (\bar Z_i))^2 \vspace*{4pt} \\
                     & \n = \n & (1 +  Dp E (\bar Z_i))^2 + Dp (1 - p)(E (\bar Z_i))^2 + Dp \var (\bar Z_i). \end{array} $$
 Then, applying~\cite[Theorem~2]{jevtic_lanchier_2020} with~$\mu = D - 1$ and~$\sigma^2 = 0$, we get
\begin{equation}
\label{eq:branching-second-3}
  \begin{array}{l}
    E (\bar X_R^2) = \bigg(1 + \displaystyle Dp \bigg(\frac{1 - \nu^R}{1 - \nu} \bigg) \bigg)^2 + Dp (1 - p) \bigg(\frac{1 - \nu^R}{1 - \nu} \bigg)^2 \vspace*{8pt} \\ \hspace*{120pt}
                             + \ \displaystyle Dp \ \frac{\nu (1 - p)}{(1 - \nu)^2} \bigg(\frac{1 - \nu^{2R - 1}}{1 - \nu} - (2R - 1) \nu^{R - 1} \bigg) \end{array}
\end{equation}
 while a simple calculation implies that
\begin{equation}
\label{eq:branching-second-4}
  \begin{array}{l}
  \displaystyle Dp (1 - p) \bigg(\frac{1 - \nu^R}{1 - \nu} \bigg)^2 +
  \displaystyle Dp \ \frac{\nu (1 - p)}{(1 - \nu)^2} \bigg(\frac{1 - \nu^{2R - 1}}{1 - \nu} - (2R - 1) \nu^{R - 1} \bigg) \vspace*{8pt} \\ \hspace*{150pt} =
  \displaystyle \frac{Dp (1 - p)}{(1 - \nu)^2} \bigg(\frac{(1 - \nu^R)(1 + \nu^{R + 1})}{1 - \nu} - 2R \nu^R \bigg). \end{array}
\end{equation}
 Combining~\eqref{eq:branching-second-3} and~\eqref{eq:branching-second-4} gives the result.
\end{proof} \\ \\
 Theorem~\ref{th:branching} directly follows from~\eqref{eq:wet-branching}, and from Lemmas~\ref{lem:branching-first} and~\ref{lem:branching-second}.


\section{Isolated vertices}
\label{sec:plarge}
 Theorem~\ref{th:branching} gives good upper bounds when the probability~$p$ is small.
 To complement this result, we now give a second set of upper bounds that are accurate when~$p$ is close to one.
\begin{theorem} --
\label{th:plarge}
 Let~$q = 1 - p$. Then,
 $$ E (S) \leq N - (N - 1) q^D \quad \hbox{and} \quad E (S^2) \leq N^2 - (N - 1)(2N - 1) q^D + (N - 1)(N - 2) q^{2D - 1}. $$
\end{theorem}
\begin{proof}
 Let~$\eta (y) = \ind \{y \in \C_x \}$ for all~$y \in \V$.
 Then, for all integers~$k \geq 1$,
\begin{equation}
\label{eq:plarge-1}
  E (S^k) = E \bigg(\sum_{y \in \V} \,\eta (y) \bigg)^k = \sum_{y_1, \ldots, y_k \in \V} E (\eta (y_1) \ \cdots \ \eta (y_k)) = \sum_{y_1, \ldots, y_k \in \V} P (x \leftrightarrow y_1, \ldots, x \leftrightarrow y_k)
\end{equation}
 where~$\leftrightarrow$ means that there is an open path.
 To estimate the last sum, we let~$B_y$ be the event that all the edges incident to~$y$ are closed.
 Using that there are exactly~$D$ edges incident to each vertex, and that there is at most one edge connecting any two different vertices, say~$y \neq z$, we get
\begin{equation}
\label{eq:plarge-2}
  P (B_y) = q^D \quad \hbox{and} \quad P (B_y \cup B_z) = P (B_y) + P (B_z) - P (B_y \cap B_z) \geq 2q^D - q^{2D - 1}.
\end{equation}
 In addition, $B_y \subset \{x \not \leftrightarrow y \}$ for all~$y \neq x$.
 This and~\eqref{eq:plarge-2} imply that
\begin{equation}
\label{eq:plarge-4}
  P (x \not \leftrightarrow y \ \hbox{or} \ x \not \leftrightarrow z) \geq \left\{\begin{array}{lcl}
    q^D & \hbox{when} & \card \{x, y, z \} = 2 \vspace*{4pt} \\
   2q^D - q^{2D - 1} & \hbox{when} & \card \{x, y, z \} = 3. \end{array} \right.
\end{equation}
 Using~\eqref{eq:plarge-1} with~$k = 1$ and~\eqref{eq:plarge-4}, we deduce that
 $$ E (S) = 1 + \sum_{y \neq x} \,(1 - P (x \not \leftrightarrow y)) \leq 1 + \sum_{y \neq x} \,(1 - q^D) = 1 + (N - 1)(1 - q^D) = N - (N - 1) q^D. $$
 Similarly, applying~\eqref{eq:plarge-1} with~$k = 2$, observing that
 $$ \begin{array}{rcl}
    \card \{(y, z) \in \V^2 : \card \{x, y, z \} = 2 \} & \n = \n & 3 (N - 1) \vspace*{4pt} \\
    \card \{(y, z) \in \V^2 : \card \{x, y, z \} = 3 \} & \n = \n & (N - 1)(N - 2), \end{array} $$
 and using~\eqref{eq:plarge-2} and~\eqref{eq:plarge-4}, we deduce that
 $$ \begin{array}{rcl}
      E (S^2) & \n \leq & \n 1 + 3 (N - 1)(1 - q^D) + (N - 1)(N - 2)(1 - 2 q^D + q^{2D - 1}) \vspace*{4pt} \\
              & \n = \n & N^2 - (N - 1)(2N - 1) q^D + (N - 1)(N - 2) q^{2D - 1}. \end{array} $$
 This completes the proof of Theorem~\ref{th:plarge}.
\end{proof}


\begin{figure}[t!]
\centering
\scalebox{0.18}{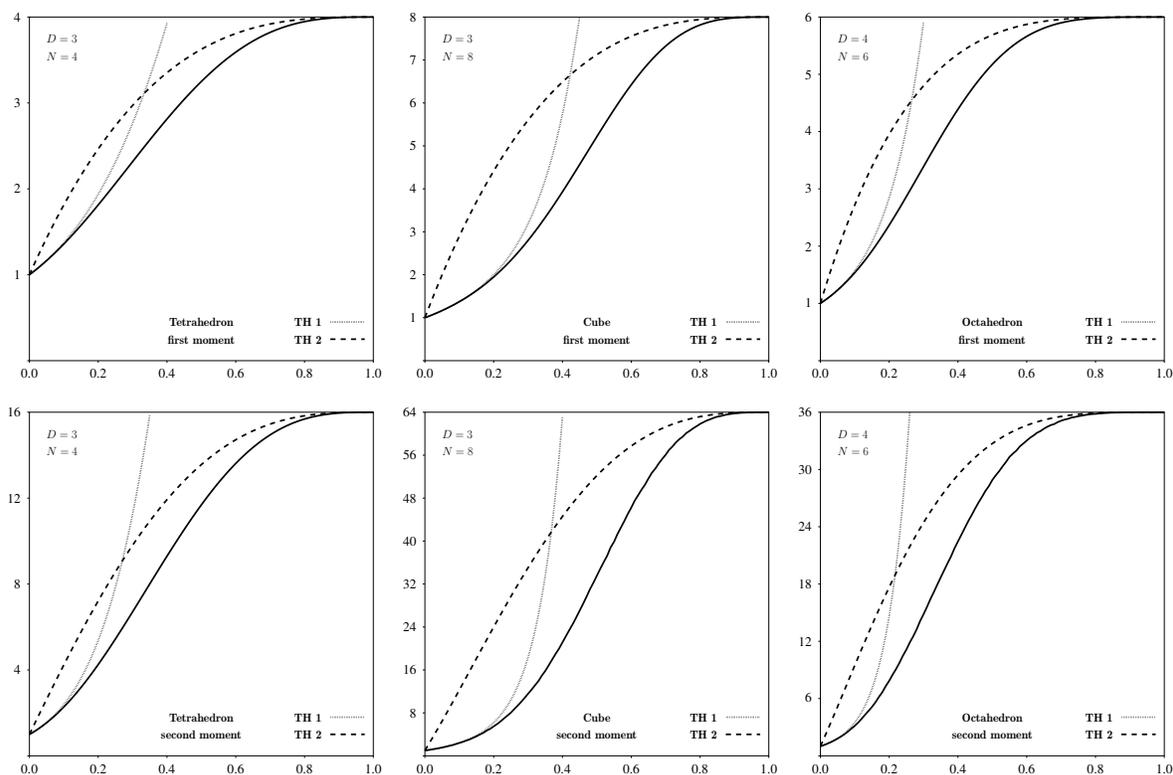}
\caption{\upshape{First and second moments of the bond percolation cluster size on the tetrahedron, the cube, and the octahedron as functions of the probability~$p$.
                  The solid lines show the moments obtained from the average of one hundred thousand independent realizations of the process.
                  The dotted lines represent the upper bound from Theorem~\ref{th:branching} and the dashed lines the upper bound from Theorem~\ref{th:plarge} for the appropriate values of~$D$ and~$N$.}}
\label{fig:solids-123}
\end{figure}

\section{Numerical examples}
\label{sec:examples}
 As an illustration, we apply our results to the first three Platonic solids~(the tetrahedron, the cube, and the octahedron) which can be viewed as worst case scenarios as they contain many cycles.
 Figure~\ref{fig:solids-123} shows the first and second moments of the cluster size obtained from numerical simulations along with the upper bounds in the theorems obtained by setting the degree~$D$ and the number of vertices~$N$ appropriately for each of the three solids.


\bibliographystyle{plain}
\bibliography{biblio.bib}

\begin{thebibliography}{1}

\bibitem{ajtai_komlos_szemeredi_1982}
M.~Ajtai, J.~Koml\'{o}s, and E.~Szemer\'{e}di.
\newblock Largest random component of a {$k$}-cube.
\newblock {\em Combinatorica}, 2(1):1--7, 1982.

\bibitem{broadbent_hammersley_1957}
S.~R. Broadbent and J.~M. Hammersley.
\newblock Percolation processes. {I}. {C}rystals and mazes.
\newblock {\em Proc. Cambridge Philos. Soc.}, 53:629--641, 1957.

\bibitem{erdos_renyi_1959}
P.~Erd\H{o}s and A.~R\'{e}nyi.
\newblock On random graphs. {I}.
\newblock {\em Publ. Math. Debrecen}, 6:290--297, 1959.

\bibitem{grimmett_1999}
G.~Grimmett.
\newblock {\em Percolation}, volume 321 of {\em Grundlehren der Mathematischen
  Wissenschaften}.
\newblock Springer-Verlag, Berlin, second edition, 1999.

\bibitem{jevtic_lanchier_2020}
P.~Jevti\'{c} and N.~Lanchier.
\newblock Dynamic structural percolation model of loss distribution for cyber
  risk of small and medium-sized enterprises for tree-based {LAN} topology.
\newblock {\em Insurance Math. Econom.}, 91:209--223, 2020.

\bibitem{jevtic_lanchier_lasalle_2020}
P.~Jevti\'{c}, N.~Lanchier, and A.~La~Salle.
\newblock First and second moments of the size distribution of bond percolation
  clusters on rings, paths and stars.
\newblock {\em Statist. Probab. Lett.}, 161:108714, 6, 2020.

\bibitem{kott_linkov_2019}
A.~Kott and I.~Linkov.
\newblock {\em Cyber resilience of systems and networks}.
\newblock Springer International Publishing, 2019.

\end{thebibliography}

\end{document}